\documentclass[12pt]{article}
\usepackage{amsmath,amssymb,color,amscd}
\usepackage{xspace}

\newcommand{\Var}{{\rm{Var}_{\mathbb{C}}}}

\def\1{\underline{1}}

\def\P{\mathbb P}

\def\Z{{\mathbb Z}}

\def\C{{\mathbb C}}

\def\B{{\mathcal B}}
\def\A{{\mathcal A}}

\def\CP{\mathbb C\mathbb P}

\def\wtimes{\widehat{\times}}

\newtheorem{theorem}{Theorem}
\newtheorem{statement}{Statement}

\newtheorem{definition}{Definition}

\newenvironment{proof}
{\noindent{\bf Proof\/}.}{{ $\square$}\smallskip\par}

\title{Grothendieck ring of pairs of quasi-projective varieties
\footnote{Math. Subject Class.: 18F30. Keywords: 
complex quasi-projective varieties, Grothendieck rings,
lambda-structure, power structure.}
}

\author{S.M.~Gusein-Zade \thanks{The work of the first author (Sections~\ref{sec:lambda} and~\ref{sec:Example})
was supported by the grant 21-11-00080 of the Russian Science Foundation.
Address: Moscow State University, Faculty
of Mechanics and Mathematics, 
Moscow Center for Fundamental and Applied Mathematics,
GSP-1, Moscow, 119991, Russia \&
National Research University ``Higher School of Economics'',
Usacheva street 6, Moscow, 119048, Russia. E-mail:
sabir\symbol{'100}mccme.ru} \and I.~Luengo \thanks{ The last two authors were partially
supported by a competitive Spanish national grant MTM PID2020-114750GB-C32.
Address:  ICMAT (CSIC-UAM-UC3M-UCM), Dept. of Algebra, Geometry and Topology, Complutense University of Madrid,
Plaza de Ciencias 3, Madrid, 28040, Spain.
E-mail: iluengo\symbol{'100}mat.ucm.es} \and
A.~Melle-Hern\'andez \thanks{Address:  Instituto de Matem\'atica Interdisciplinar (IMI),
Dept. of Algebra, Geometry and Topology, Complutense University of Madrid,
Plaza de Ciencias 3, Madrid, 28040, Spain. E-mail: amelle\symbol{'100}mat.ucm.es}}
\date{}
\begin{document}
\def\eps{\varepsilon}

\maketitle

\begin{abstract}
We define a Grothendieck ring of pairs of complex quasi-projective varieties (that is a variety and a subvariety).
We describe $\lambda$-structures
and a power structure on/over this ring.
We show that the conjectual symmetric power
of the projective line with several orbifold
points described by A.~Fonarev is consistent
with the symmetric power of this line with
points as a pair of varieties.
\end{abstract}

\section{Introduction}\label{sec:Intro}
We consider the Grothendieck ring $K_0^{\rm pairs}(\Var)$ of pairs of complex quasi-projective varieties and define natural
$\lambda$-structures and a power structure on/over it.
One motivation for that came from the intention to define
the Grothendieck ring of (quasi-projective) orbifolds (and to
study $\lambda$-structures and power structures on/over it).
The simplest (non-trivial) orbifold is a complex cuve $C$
with several orbifold points $a_1$, \dots, $a_s$ which can
be considered as the pair $(C, \{a_1,\ldots, a_s\})$ of
varieties (of dimension 1 and 0 respectively). Submanifolds in affine or projective spaces also can be considered as pairs.

Grothendieck rings of varieties (with additional structures)
often are endowed by $\lambda$-structures. Natural 
$\lambda$-structures on them are defined by analogues of
the Kapranov zeta function (the generating series of the
classes of the symmetric products) and of the generating series
of the classes of the configuration spaces. We consider natural
analogues of these series for pairs of varieties.
A $\lambda$-structure on a ring defines a power structure over it
(see~\cite{GLM-MRL}). Power structures over Grothendieck rings of varieties with additional structures turn out to be
useful to write down and/or to prove statements about generating series of classes of some varieties: see, e.g., \cite{GLM-Michigan}, \cite{GLM-Steklov}, \cite{Relative}.
We show that two natural $\lambda$-structures on $K_0^{\rm pairs}(\Var)$ define one and the same
power structure over it and give its geometric description.

\section{The Grothendieck ring of pairs of varieties}
\label{sec:Groth_ring}
A pair of (complex quasi-projective) varieties is a pair
$(X,Y)$ consisting of a complex quasi-projective variety $X$
and a Zariski locally closed
subvariety $Y$ in it. The variety $X$ will be called {\em the ambient variety}, $Y$ will be called {\em the subvariety} of the pair.

The {\em Grothendieck semiring $S_0^{\rm pairs}(\Var)$ of
pairs of complex quasi-projective varieties} is the semigroup
generated by the isomorphism classes $[(X,Y)]$ of pairs modulo
the relation
$$
[(X,Y)]=[(Z,Y\cap Z)]+[(X\setminus Z,Y\cap(X\setminus Z))]
$$
for a Zariski closed subvariety $Z\subset X$; the multiplication in $S_0^{\rm pairs}(\Var)$ is defined by
$$
[(X_1,Y_1)]\wtimes[(X_2,Y_2)]=
[(X_1\times X_2,(X_1\times Y_2)\cup(Y_1\times X_2))]\,.
$$

The {\em Grothendieck ring $K_0^{\rm pairs}(\Var)$} (of
pairs of complex quasi-projective varieties) is the (abelian)
group
generated by the isomorphism classes of pairs with the
same relation and the same multiplication.

The unit $1$ in this ring (and also in the semiring) is represented by the pair $({\rm pt},\emptyset)$ where 
${\rm pt}$ is a one point set (${\rm Spec\,}\C$).

\section{$\lambda$-structures on $K_0^{\rm pairs}(\Var)$}
\label{sec:lambda}
A {\em $\lambda$-structure} on a ring $R$ (often called a
pre-$\lambda$-structure) is defined by a series 
$\lambda_a(t)\in 1+t\cdot R[[t]]$ given for all $a\in R$
such that 
$\lambda_a(t)=1+at+~terms~of~higher~order$, 
$\lambda_{a+b}(t)=\lambda_a(t)\cdot\lambda_b(t)$
(see, e.\,g.,~\cite{Knutson}).

Let $(X,Y)$ be a pair of varieies. The $n$th symmetric product
of $(X,Y)$ is $S^n(X,Y)=(X,Y)^n/S_n$, that is the pair consisting of the symmetric product $S^nX$ of $X$ with the subvariety of
(unordered) $n$-tuples $\{x_1,\ldots, x_n\}$ of points of $X$
such that (at least) one of these points belongs to $Y$.
(Here and below the Cartesian power
$(X,Y)*n$ of  pair means
$(X,Y)\widehat{times}\ldots\widehat{times}(X,Y)$, $n$ times.)
Example: $S^n1=1$, where $1=[({\rm pt},\emptyset)]$.
($(X,Y)^0=({\rm pt},\emptyset)$ and therefore $[S^0(X,Y)]=1$.)

The $n$th (unordered) configuration space of $(X,Y)$ is
$\Lambda^n(X,Y)=\left((X,Y)^n\setminus\Delta\right)/S_n$,
where $\Delta$ is the big diagonal consisting of $n$-tuples of points with at least two coinciding ones. In other words
$\Lambda^n(X,Y)$ is the pair consisting of the usual configuration space $\Lambda^nX=(X\setminus\Delta)/S_n)$ of $X$ with the subspace of
(unordered) $n$-tuples $\{x_1,\ldots, x_n\}$ of different points of $X$
such that (at least) one of them belongs to $Y$.
Example: $\Lambda^n1=0$ for $n\ge 2$.

\begin{definition}
 The Kapranov zeta function of $[(X,Y)]$ is the series
 $$
 \zeta_{X,Y}(t)=1+\sum_{n=1}^{\infty}[S^n(X,Y)]\cdot t^n
 \in 1+t\cdot K_0^{\rm pairs}(\Var)[[t]]\,.
 $$
\end{definition}

Example: $\zeta_1(t)=1+t+t^2+\ldots=(1-t)^{-1}$.

\begin{statement}\label{st:Kapranov}
 The Kapranov zeta function $\zeta_{X,Y}(t)$ defines a
 $\lambda$-structure on $K_0^{\rm pairs}(\Var)[[t]]$.
\end{statement}

\begin{proof}
 We has to show that, for two pairs $(X_1,Y_1)$ and $(X_2,Y_2)$, one has
 \begin{equation}\label{eqn:Kapranov-lambda}
 \zeta_{X_1\sqcup X_2,Y_1\sqcup Y_2}(t)=
 \zeta_{X_1,Y_1}(t)\cdot \zeta_{X_2,Y_2}(t)\,.
 \end{equation}
 The coefficient of $t^n$ in the LHS of~(\ref{eqn:Kapranov-lambda})
 is the pair consisting of the space of unordered $n$-tuples
 $\{x_1,\ldots,x_n \}$ of points of $X_1 \sqcup X_2$
 (that is $S^n(X_1 \sqcup X_2)$) with the subvariety of
 tuples $K=\{x_1,\ldots,x_n \}$ such that at least one of
 the pointx $x_i$ belongs to $Y_1\sqcup Y_2$.
 This means that this space of pairs is the union over $s$, $0\le s\le n$,
 of the spaces of pairs consisting of the products
 $S^s X_1\times S^{n-s} X_2$
 with the subvarieties of them consisting of an (unordered) $s$-tuple $K_1$
 of points of $X_1$ and an (unordered) $(n-s)$-tuple $K_2$ of points of $X_2$ such that at least one point of one of them
 belongs to the corresponding $Y_i$, $i=1 \text{ or }2$. 
 The coefficient of $t^n$ in the RHS of~(\ref{eqn:Kapranov-lambda}) is $\left[\bigsqcup_{s=0}^n
 S^s(X_1,Y_1)\widehat{\times} S^{n-s}(X_2,Y_2)\right]$
 and has the same description as the LHS. 
\end{proof}

Let
 $$
 \lambda_{X,Y}(t)=1+\sum_{n=1}^{\infty}[\Lambda^n(X,Y)]\cdot t^n
 \in 1+t\cdot K_0^{\rm pairs}(\Var)[[t]]
 $$
 be the generating series of the classes of the configuration spaces $\Lambda^n(X,Y)$. Example: $\lambda_1(t)=1+t$.
 
 \begin{statement}\label{st:lambda-series}
 The series $\lambda_{X,Y}(t)$ defines a
 $\lambda$-structure on $K_0^{\rm pairs}(\Var)[[t]]$.
\end{statement}

The {\bf proof} is literally the same as the one of
Statement~\ref{st:Kapranov} with the only difference that
the tuples under consideration should consist of different points.

\section{Example: the symmetric product of the projective line with several distinguished points}\label{sec:Example}
Let us consider the pair $(\CP^1, \{a_1, \ldots, a_s\})$ ($a_i\ne a_j$ for $i\ne j$) consisting of the projective line
and an $s$-point subset of it. We shall describe its $n$th symmetric  product.

The $n$the symmetric product of the projective line $\CP^1$ with the 
coordinates $(u:v)$ is the $n$-dimensional projective space
$\CP^n$. An isomorphism between $S^n(\CP^1)$ and $\CP^n$ can be
established by the following map. Let $(p_0: p_1:\ldots: p_n)$
be coordinates on $\CP^n$. A point of $S^n(\CP^1)$ is
represented by an $n$-tuple of points $\{x_1,\ldots, x_n)$ of
$\CP^1$, $x_i=(u_i:v_j)$. The point in $\CP^n$ corresponding to
this $n$-tuple is
\begin{eqnarray*}
 p_0&=&+u_1u_2\ldots u_n,\\
 p_1&=&-\sum_{1\le i\le n} u_1u_2\ldots \widehat{u_i}v_i\ldots u_n,\\
 p_2&=&+\sum_{1\le i<j\le n} u_1u_2\ldots \widehat{u_i}v_i\ldots \widehat{u_j}v_j\ldots u_n,\\
 &{\ }&\ldots\ldots\ldots\\
 p_n&=&\pm\ v_1v_2\ldots v_n,
\end{eqnarray*}
where (as usual) $\widehat{\cdot}$ means that the corresponding term is
excluded. (This means that $z_i=u_i/v_i$ are the roots of the
polynomial $P(z)=p_0+p_1z+p_2z^2+\ldots+p_nz^n$.)
Let $a_i=(x_i:1)$ The subset $Y$ in the pair 
$S_n(\CP^1, \{a_1, \ldots, a_s\})=(\CP^n,Y)$ consists of the points $(p_0: p_1:\ldots: p_n)\in\CP^n$ such that one of the
roots of the polynomial $p_0+p_1z+p_2z^2+\ldots+p_nz^n$
is one of $x_i$. This means that $Y$ is the union over $i=1, \ldots, s$ of the subsets of the points $(p_0: p_1:\ldots: p_n)\in \CP^n$ such that $p_0+p_1x_i+p_2x_i^2+\ldots+p_nx_i^n=0$.
These are $s$ hyperplanes in $\CP^n$ in general position.
This gives a partial explanation of the (conjectual) description of the
orbifold $S^n(\CP^1, \{a_1, \ldots, a_s\})$ with $s$ points of order
$2$ in~\cite{Fonarev}.

\section{The geometric power structure over the ring
$K_0^{\rm pairs}(\Var)[[t]]$}\label{sec:Power}
A power structure over a ring $R$ is a method to give sense to
an expression of the form $(A(t))^m$, where
$A(t)=1+a_1t+a_2t^2+\ldots\in 1+t\cdot R[[t]]$ (that is
$a_i\in R$) and $m\in R$, as a series from $1+t\cdot R[[t]]$ (in other words it is a map from
$(1+t\cdot R[[t]])\times R$ to $1+t\cdot R[[t]]$:
$(A(t),m)\mapsto (A(t))^m$ so that all the usual properties of
the exponential function hold the main of them (those listed in~\cite{GLM-MRL}; later some more requirements (obvious in the case under consideration) were added:~\cite{GLM-Steklov}) being
\begin{enumerate}
\item[1)] $\left(A(t)\right)^0=1$;
\item[2)] $\left(A(t)\right)^1=A(t)$;
\item[3)] $\left(A^{(1)}(t)\cdot A^{(2)}(t)\right)^m=\left(A^{(1)}(t)\right)^m\cdot \left(A^{(2)}(t)\right)^m$;
\item[4)] $\left(A(t)\right)^{m_1+m_2}=\left(A(t)\right)^{m_1}\cdot \left(A(t)\right)^{m_2}$;
\item[5)] $\left(A(t)\right)^{m_1m_2}=\left(\left(A(t)\right)^{m_2}\right)^{m_1}$.
\end{enumerate}
A $\lambda$-structure on a ring $R$ defines a power structure over it. On the other hand, in general, there are many 
$\lambda$-structures corresponding to one and the same power structure.

A power structure over the ring $K_0^{\rm pairs}(\Var)$
is called effective if, in fact, it is defined over the
semiring $S_0^{\rm pairs}(\Var)$. This means that,
if the coefficients $a_i$ of the series $A(t)$ and the exponent $m$ are the classes of pairs of varieties (not differences of such classes), then all the coefficients of the series
$(A(t))^m$ are also represented by classes of pairs.

\begin{theorem}
The $\lambda$-structures on the ring $K_0^{\rm pairs}(\Var)[[t]]$ described in Section~\ref{sec:lambda} define one and the same 
power structure over it, which is effective.
\end{theorem}

\begin{proof}
To prove that, we shall give a geometric
description of this power structure. This means that we shall describe the
coefficient at $t^k$ in the series
$(A(t))^{[(M,N)]}$, where $A(t)=1+\sum_{k=1}^{\infty}
[(A_k,B_k)]\cdot t^k$, as the class of a pair of
quasi-projective varieties.
This will mean that the power structure is effective (after it
is proved that the description really gives a power structure).
The fact that this power structure corresponds both to the
Kapranov zeta function $\zeta_{X,Y}(t)$ an to the generating series
$\lambda_{X,Y}(t)$ of the classes of the configuration spaces
will follow from the equations 
\begin{eqnarray}
 \zeta_{X,Y}(t)&=&(1-t)^{-[(X,Y)]}\quad\text{and}\label{eqn:powerKapranov}\\
 \lambda_{X,Y}(t)&=&(1+t)^{[(X,Y)]}\label{eqn:powerLambda}
\end{eqnarray}
(in the sense of the power structure).

The description is as follows. Let
$\A^*=\bigsqcup_{i=1}^{\infty}A_i$,
$\A:=\A^*\cup\{{\rm pt}\}$,
$\B=\B^*=\bigsqcup_{i=1}^{\infty}B_i$. 
Let $i$ be the tautological (integer valued) function on $\A$ and on $\B$ which sends, for $i>0$, $A_i$ and $B_i$ to $i\in\Z$ and the point
${\rm pt}\in \A$
to $0$.

The coefficient at $t^{n}$ in the series ${A}(t)^{[(M,N)]}$ is represented by the following pair of varieties.
The ambient space is
the configuration space of pairs $(K,\varphi)$, where $K$ is a finite subset of the variety $M$ and $\varphi$ is a map from $K$ to $\A^*$
such that $\sum\limits_{x\in K}i(\varphi(x))=n$.
The subvariety of the pair is the set of points of the ambient space
such that
either (at least) one of the points of $K$ lies in $N$
or the image of (at least) one point of $K$ is in $B^*$.

To describe such a configuration space as a quasi-projective variety one can write it as
\begin{equation}\label{eqn:geom}
\sum_{{\mathbf k}:\,\sum {i}\,k_i=n}
\left[
\left(\left(
(\prod_{i} (M,N)^{k_i})
\setminus\Delta
\right)
\widehat{\times}\prod_i (A_i,B_i)^{k_i}\right)\left/\prod_i S_{k_i}\right.
\right]\,,
\end{equation}
where ${\mathbf k}=\{k_1, k_2, \ldots\}$, $k_i\in\Z_{\ge 0}$ and $\Delta$ is the ``large diagonal" in $(M,N)^{\Sigma k_i}$ which consists of $(\sum k_i)$-tuples of points of $M$ with at least two coinciding ones; the permutation group $S_{k_i}$ acts by permuting corresponding $k_i$ factors in
$\prod\limits_i (M,N)^{k_i}\supset (\prod\limits_{i} (M,N)^{k_i})\setminus\Delta$ and the pairs $(A_i,B_i)$ in $(A_i,B_i)^{k_i}$ simultaneously (the connection between this formula and the description above is clear).

To prove that the described construction really
gives a power structure over $K_0^{\rm pairs}(\Var)$, one has to verify the conditions 1) -- 5). The first two are obvious.
The ambient spaces in the summands in Equation~(\ref{eqn:geom}) coincide with
the corresponding summands for the
expression $\left(1+\sum_i[A_i]t^i\right)^{[M]}$
in~\cite[Equation~(1)]{GLM-MRL} (i.e.~for the (usual) power structure over the Grothendieck ring $K_0(\Var)$ of complex quasi-projective varieties).
Therefore the corresponding equalities
for them can be found in~\cite{GLM-MRL}.
Thus it is necessary to verify the conditions 3) -- 5) only
for the subvarieties of the pairs.

Condition 3) ($A^{(j)}(t)=1+\sum\limits_{i=1}^{\infty}[(A^{(j)}_i,B^{(j)}_i)]t^i$, $j=1,2$, $m=[(M,N)]$).
The map (an isomorphism) from the right hand
side to the left hand side on the ambient
spaces is the following. 
If $A(t)=A^{(1)}(t)\widehat{times} A^{(2)}(t)
=1+\sum_{i=1}^{\infty}(A_i,B_i)$, then
$\A=\A^{(1)}\times \A^{(2)}$.
If $K_1$ and $K_2$
are finite subsets of $M$ with maps
$\psi_1$ and $\psi_2$ to $A^{(1)}$ and
$A^{(1)}$ respectively (these data define
a point from the union of the coefficients
of the right hand side), then the corresponding (finite) subset $K$ is the
union $K_1\cup K_2$ with the map $\psi$ to
$A^*$ defined by
$\psi(x)=(\psi_1(x),\psi_2(x))$, where,
if $x\notin K_1$ (respectively if 
$x\notin K_2$), one assumes $\psi_1(x)={\rm pt}$ ($\psi_2(x)={\rm pt}$ respectively).
If a point of $K$ lies in $N$, but either
a point of $K_1$ or a point of $K_2$ lies there. If all the points of $K$ are not
in $N$, then one of its points maps to
$\B=(\A^{(1)}\times \B^{(2)})\cup
(\B^{(1)}\times \A^{(2)})$, then either
its image under $\psi_1$ lies in $\B^{(1)}$
or its image under $\psi_2$ lies in $\B^{(2)}$. This means that the corresponding
point of the union of the coefficients
of $(A(t))^{[(M,N)]}$ lies in the image of the subspace
in the product of the unions of coefficients
of $(A^{(1)}(t))^{[(M,N)]}$ and of $(A^{(2)}(t))^{[(M,N)]}$ and vice versa.

Condition 4) ($m_1=[(M_1,N_1)]$,
$m_2=[(M_2,N_2)]$).
The map (an isomorphism) from the right hand
side to the left hand side on the ambient
spaces is the following. If $K_1$
(respectively $K_2$) is a finite subset
of $M_1$ (respectively of $M_2$) with
maps $\psi_1:K_1\to\A^*$ and
$\psi_2:K_1=2\to\A^*$ (these data define
a point from the union of the coefficients
of the right hand side),
then $K\subset M_1\cup M_2$ is the union
of (disjoint) $K_1$ and $K_2$ with the
natural map to $\A^*$.
If a point of $K$ lies in $N_1\cup N_2$,
then either a point of $K_1$ lies in $N_1$
or a point of $K_2$ lies in $N_2$.
If all the points of $K$ are outside of
$N_1\cup N_2$, but the image of one of them
lies in $\B$, then either of a point of $K_1$ or a point of $K_2$ maps to $\B$.
This means that the corresponding
point of the union of the coefficients
of $(A(t))^{[(M_1\cup M_2,N_1\cup N_2)]}$ lies in the image of the subspace
in the product of the unions of coefficients
of $(A(t))^{[(M_1,N_1)]}$ and of $(A(t))^{[(M_2,N_2)]}$ and vice versa.

Condition 5) ($m_1=[(M_1,N_1)]$,
$m_2=[(M_2,N_2)]$).
The map (an isomorphism) from the left hand
side to the right hand side on the ambient
spaces (in this case it is more convenient to describe the map in this direction) is the following. 
If $K$ is a finite subset of $M_1\times M_2$
with a map $\psi$ to 
$\A^*$ (these data define
a point from the union of the coefficients
of the left hand side), 
then the corresponding (finite) subset $K_1$
of $M_1$ is the projection $\pi_1(K)$
to the first factor.
For each point $y$ from $K_1$ the corresponding point of the union of the coefficients of $\left(A(t)\right)^{[M_2]}$
(i.e.\ the point to which $y$ is mapped)
is represented by the subset
$\pi_2\pi_1^{-1}(y)$ of $M_2$ with the map
to $\A^*$ being the corresponding restriction. If a point $x\in K$ lies 
in $M_1\times N_2\cup M_2\times N_1$,
then either its projection to $M_1$ lies
in $N_1$ or its projection to $M_2$ lies
in $N_2$. If this is not the case, but
$\psi(x)$ lies in $\B$, then the corresponding point $\pi_1(x)\in M_2$
is mapped to the subspace in the union of
the coefficients of $(A(t))^{[(M_2,N_2)]}$.
This means that the corresponding
point of the union of the coefficients
of $\left((A(t))^{[(M_2,N_2)]})\right)^{[(M_1,N_1)]}$ lies in the image of the subspace of the union of coefficients
of $(A(t))^{[(M_1,N_1)]\widehat{\times}[(M_1,N_1)]}$ and vice versa.
\end{proof}

\end{document}